\documentclass{amsart}
\usepackage {amsmath, amscd}
\usepackage {amssymb}
\usepackage {bm}
\usepackage{fourier}

\usepackage[active]{srcltx}

\def\<{\langle}
\def\>{\rangle}

\numberwithin{equation}{section}

\def\DD{{\mathcal D}}
\def\EE{{\mathcal E}}
\def\FF{{\mathcal F}}
\def\GG{{\mathcal G}}
\def\HH{{\mathcal H}}

\def\KK{{\mathcal K}}
\def\LL{{\mathcal L}}

\def\RR{{\mathcal R}}

\def\VV{{\mathcal V}}

\def\XX{{\mathcal X}}
\def\YY{{\mathcal Y}}

\def\bbC{\mathbb{C}}
\def\bbZ{\mathbb{Z}}
\def\bbN{\mathbb{N}}

\def\bbD{\mathbb{D}}

\newcommand{\rank}{\mathop{\rm rank}}

\newcommand{\Range}{\RR}

\newcommand{\Prec}{\preccurlyeq}
\newcommand{\Red}{\prec}

\newtheorem{lemma}{Lemma}[section]
\newtheorem{theorem}[lemma]{Theorem}
\newtheorem{corollary}[lemma]{Corollary}

\theoremstyle{definition}

\newtheorem{example}[lemma]{Example}

\title{On a preorder relation for contractions}

\author{Dan Timotin}

\address{Dan Timotin, Institute of Mathematics Simion Stoilow of the Romanian Academy, P.O. Box 1-764, Bucharest 014700,
	Romania}
\email{Dan.Timotin@imar.ro}



\begin{document}

\begin{abstract}
	An order relation  for contractions on a Hilbert space can be introduced  by stating that  $A\Prec B$ if and only if $A$  is unitarily equivalent to the restriction of $B$ to an invariant subspace. We discuss the equivalence classes associated to this relation, and identify cases in which they coincide with classes of unitary equivalence. The results extend those for completely nonunitary partial isometries obtained by Garcia, Martin, and Ross.
\end{abstract}
\maketitle

\section{Introduction}

The starting point for this note is the paper~\cite{GMR}, which makes a detailed analysis of the class of
completely nonunitary partial isometries on a complex separable Hilbert space with equal defect indices (finite or infinite). 
This analysis, having its original source in a classical paper of Livsic~\cite{L},  also draws on subsequent work in~\cite{AMR, HML, M}.
Among other things, three preorder relations are introduced in~\cite{GMR}; 
each of these  induces an equivalence relation, and one of the questions addressed is when the classes of equivalence are precisely the classes of unitary equivalence. It is shown that the question can be rephrased as the existence of certain multipliers between model spaces (in the sense of~\cite{BR} or~\cite{NF}), and partial results are obtained.

We will concentrate on one of these preorder relations, that may be easily described: $A\Prec B$ if there exists a subspace $\YY\subset\HH_B$, invariant with respect to $B$, such that $A$ is unitarily equivalent to $B|\YY$.
The relation has an obvious extension to all contractions (even to general bounded operators, but we will stick to contractions).
 One of the tools in ~\cite{GMR} is the characteristic function of a partial isometry, and  a comprehensive theory of Sz.-Nagy--Foias~\cite{NF} extends the notion of characteristic function to all completely non unitary contractions.  On the other hand, the structure of unitary operators is well known by multiplicity theory~\cite{Ha}.
 
  Considering thus the relation $\Prec$ in the context of general contractions, we give in this paper a rather comprehensive answer to the question of deciding which corresponding classes of equivalence   coincide with classes of unitary equivalence.  Theorems 9.1 and 9.2 from~\cite{GMR} are thus significantly generalized. In short, the statement is true  if some finite multiplicity condition is satisfied, while it fails in general.

The plan of the paper is the following. We start with the necessary preliminaries about contractions. In Section~\ref{se:main problem} we introduce the preorder  relation, obtain some simple results, and state the main problem.   Section~\ref{se: n finite} investigates it for completely nonunitary contractions, using the Sz.-Nagy--Foias theory, which links invariant subspaces of a contraction to factorizations of its characteristic function. The main result here is Theorem~\ref{th:general n fnite}, where  finiteness of the defect spaces plays an essential role; it is shown by a counterexample  that this assumption cannot be dropped. Finally, in Section~\ref{se:general} one discusses the general situation of contractions that may have a unitary part, obtaining a partial  solution in Theorem~\ref{th:a more general result}.

\section{Preliminaries}\label{se:prelim}

We deal with contractions which may act on different Hilbert spaces; usually $A\in\LL(\HH_A)$, $B\in\LL(\HH_B)$. 
The kernel and the range of a contraction $A$ are denoted by $\ker A$ and $\RR(A)$ respectively.

For a  contraction  $A\in\LL(\HH_A)$ there is a unique decomposition $\HH_A=\HH_A^c\oplus \HH_A^u $ with respect to which  $A=A_c\oplus A_u$,  where $A_u$ is unitary, while $A_c$ is \emph{completely nonunitary} (c.n.u.), meaning that there is no non-zero subspace that reduces it to a unitary operator. Further on, the unitary part may be further decomposed as $\HH_A^u=\HH_A^a\oplus\HH_A^s$, with respect to which $A_u=A_a\oplus A_s$, with $A_a$ absolutely continuous and $A_s$ singular with respect to Lebesgue measure.

If $\YY$ is a subspace of a Hilbert space $\XX$, then $P_\YY$ denotes the orthogonal projection onto $\YY$. Occasionally we will write $P^\XX_\YY$ when the domain is relevant.

For any contraction $A\in\LL(\HH_A)$ a \emph{minimal unitary dilation} is a unitary operator $U\in\LL(\KK_A)$ with $\HH_A\subset \KK_A$ such that $A^n=P_{\HH_A}U^n|\HH_A$ for all $n\in \bbN$ and $\KK_A=\bigvee_{k=-\infty}^\infty U^k\HH_A$. The minimal unitary dilation is  uniquely defined up to unitary equivalence. 
In fact, only the c.n.u. part of a contraction has to be actually dilated: with the above notations, if $U_c$ is a minimal unitary dilation of $A_c$, then $U=U_c\oplus A_u$ is a minimal unitary dilation of $A$. Moreover, the minimal unitary dilation of a c.n.u. contraction is absolutely continuous with respect to Lebesgue measure.

If $A$ is a contraction, then the \emph{defect operator} $D_A$ is defined by $(I-A^*A)^{1/2}$, and the \emph{defect space} is $\DD_A=\overline{\Range(D_A)}$. 

\begin{lemma}\label{le:dimensions defects}
Suppose $B$ is a completely nonunitary contraction and	 $A\Prec B$. Then $\dim\DD_A\le \dim\DD_B$ and $\dim\DD_{A^*}\le\dim\DD_{B}+\dim\DD_{B^*} $.
\end{lemma}

\begin{proof}
	The first statement is immediate; the second follows, for instance, from ~\cite[Proposition VII.3.6]{NF}.
\end{proof}

The theory of contractions often splits in two.  Unitary operators are precisely described by spectral multiplicity theory (see, for instance,~\cite{Ha}). This can be considered standard material, which does not need further discussion. On the other hand, c.n.u. contractions form the object of the Sz.-Nagy--Foias theory, which may be found in~\cite{NF}. We will  now briefly present the part of the latter that is relevant for our purposes.

A central role is played by contractively valued analytic functions $\Xi(\lambda):\EE\to\EE_*$, where  $\lambda\in \bbD$; that is, $\Xi(\lambda)\in\LL(\EE, \EE_*)$ and $\|\Xi(\lambda)\|\le 1$ for all $\lambda\in\bbD$, and the dependence in $\lambda$ is analytic.
 Such a function is called \emph{pure} if $\|\Xi(\lambda)x\|<\|x\|$ for all $\lambda\in\bbD$ and $0\not=x\in\EE$. 
 For a general contractively valued analytic function $\Xi(\lambda)$ 
 there exist  orthogonal decompositions $\EE=\EE'\oplus\EE''$, $\EE_*=\EE'_*\oplus\EE''_*$ such that $\Xi(\lambda)=\Xi^p(\lambda)\oplus \Xi_0$, where $\Xi^p(\lambda)$ is pure, while $\Xi_0$ is a unitary constant that is independent of $\lambda$; then $\Xi^p(\lambda)$ is called the \emph{pure part} of $\Xi(\lambda)$.  Two operator valued analytic functions $\Xi(\lambda), \Xi'(\lambda)$ are said to \emph{coincide}  if there are unitaries $\tau, \tau'$ such that $\Xi'(\lambda)=\tau\Xi(\lambda)\tau'$ for all $\lambda\in\bbD$.

The characteristic function of a completely nonunitary contraction $T\in\LL(\HH)$ is the contractively valued analytic function $\Theta_T(\lambda):\DD_T\to \DD_{T^*}$, defined by
\[
\Theta_T(\lambda)=-T+\lambda D_{T^*}(I-\lambda T^*)^{-1}D_T|\DD_T,\quad \lambda\in\bbD.
\]
This function can be shown to be  pure.
%
It is a complete unitary invariant for c.n.u. contractions: two c.n.u. contractions are unitarily equivalent if and only if their characteristic functions coincide. This fact and its developments constitute the model theory for c.n.u. contractions, for which we refer to~\cite{NF}; a different, but equivalent, approach can be found in~\cite{BR}.


It is shown in~\cite[Chapter VII]{NF}  that invariant subspaces of contractions correspond to certain special factorizations of the characteristic function. We briefly present in the sequel the facts that we need; references are~\cite{NF1, NF}. If $T_1:\EE_1\to\EE_2$, $T_2:\EE_2\to \EE_3$ are contractions, the factorization $T_2T_1$ is called \emph{regular} if
\[
D_{T_2}\EE_2\cap D_{T_1^*}\EE_2=\{0\}.
\]
The next lemma gathers some immediate properties of regular factorizations, that we will use below in Section~\ref{se: n finite}.

\begin{lemma}\label{le:regular properties}
	{\rm (i) } With the above notations, if $\dim\EE_2>0$ and $T_1=0$, then $T_2T_1$ is regular iff $T_2$ is an isometry.
	
	 {\rm (ii) } The factorization $T_2'T_1'\oplus T_2''T_1''=(T_2'\oplus T_2'')(T_1'\oplus T_1'')$ is regular if and only if $T_2'T_1'$ and $T_2''T_1''$ are both regular.
\end{lemma}

If $\Theta_1(\lambda):\EE_1\to\EE_2$, $\Theta_2(\lambda):\EE_2\to\EE_3$ are contractively valued analytic functions, then
 the factorization
\[
	\Theta(\lambda)=\Theta_2(\lambda)\Theta_1(\lambda)
\]
 is called regular if  $\Theta_2(e^{it})\Theta_1(e^{it})$ is regular for almost all $t$.

We will use then the next theorem,  proved in Sections~1 and~2 of~\cite[Chapter VII]{NF}.

\begin{theorem}\label{th:inv subspaces - factorizations}
	Suppose $T$ is a c.n.u. contraction with characteristic function $\Theta$. If $\HH'\subset \HH$ is an invariant subspace with respect to $T$ and the decomposition of $T$ with respect to $\HH'\oplus \HH'{}^\perp$ is
	\[
	T= \begin{pmatrix}
	T_1 & X\\ 0 & T_2
	\end{pmatrix},
	\]
	then there is an associated regular factorization
	\begin{equation}\label{eq:factorization}
	\Theta(\lambda)=\Theta_2(\lambda)\Theta_1(\lambda)
	\end{equation}
	of the characteristic function $\Theta$, such that the characteristic function of $T_i$ coincides with the pure part of $\Theta_i$ for $i=1,2$. 
\end{theorem}

In general the factors in~\eqref{eq:factorization} may have a constant unitary part, and so do not necessarily coincide with the characteristic functions of $T_1$ and $T_2$, respectively, which are always pure. However, we have the following lemma.

\begin{lemma}\label{le:factorization for n finite}
	With the above notations, suppose $\dim\DD_T=\dim\DD_{T_1}<\infty$. Then in the factorization~\eqref{eq:factorization}   $\Theta_1$ coincides with the characteristic function of $T_1$. 	
\end{lemma}

\begin{proof}

	From the factorization~\eqref{eq:factorization} obtained in Theorem~\ref{th:inv subspaces - factorizations} it follows that the domains of definition of $\Theta(\lambda)$ and $\Theta_1(\lambda)$ coincide, and so have both dimension $\dim\DD_T<\infty$. If $\Theta_1$ had a constant unitary part, then the domain of definition of its pure part would have  strictly smaller dimension. But this pure part coincides with the characteristic function of $T_1$, whose domain $\DD_{T_1}$ has dimension  $\dim\DD_{T_1}=\dim\DD_T$. The contradiction obtained shows that $\Theta_1$ is pure, and so it coincides with the characteristic function of $T_1$.
\end{proof}

\section{Preorder relations and the main problem}\label{se:main problem}

We define two preorder relations on  contractions $A\in\LL(\HH_A)$, $B\in\LL(\HH_B)$ on Hilbert spaces: 
\begin{itemize}
	\item[(i)] $A\Prec B$ if there exists a subspace $\YY\subset\HH_B$, invariant with respect to $B$, such that $A$ is unitarily equivalent to $B|\YY$.
	\item[(ii)] $A\Red B$ if there exists a subspace $\YY\subset\HH_B$, reducing  $B$, such that $A$ is unitarily equivalent to $B|\YY$.
\end{itemize}
One denotes the associated equivalence relations by $\approx$, respectively $\sim$ (so, for instance, $A\approx B$ if and only if $A\Prec B$ and $B\Prec A$).

An equivalent definition of $A\Prec B$ is to state that there exists an isometric map $\Omega:\HH_A\to \HH_B$ such that 
\begin{equation}\label{eq:Omega A=B Omega}
	\Omega A=B\Omega.
\end{equation}
We will also say that $\Omega$ \emph{implements} the relation $\Prec$.
Note that this relation coincides, for partial isometries with equal deficiency indices, with $\Prec$ defined in~\cite[Definition 7.2]{GMR} (although the formulation is slightly different).

From~\eqref{eq:Omega A=B Omega} it follows that $\Omega\HH_A$ is invariant to $B$. Definition (ii) admits a similar reformulation: $A\Red B$ if and only if there exists an isometric map $\Omega:\HH_A\to \HH_B$ such that $\Omega\HH_A$ reduces $B$ and~\eqref{eq:Omega A=B Omega} is satisfied.


It is obvious that 
\begin{center}
	$A$ and $B$  unitarily equivalent $\implies$ $A\sim B$ $\implies$ $A\approx B$. 
\end{center}
We are interested to determine whether  these implications can  be reversed. It turns out that the answer may be easily obtained for the first implication.

\begin{theorem}\label{th:solution for sim}
	If $A\sim B$, then $A$ and $B$ are unitarily equivalent.
\end{theorem}

\begin{proof}
	The proof is done by an analogue of the classical Cantor--Bernstein argument, so we will be sketchy with the details. Suppose $\Omega:\HH_A\to \HH_B$ is an isometry such that $\Omega\HH_A$ reduces $B$ and $\Omega A=B\Omega$; obviously $\HH_B\ominus \Omega\HH_A$ is also reducing for $B$. Moreover, if the subspace $\YY\subset\HH_A$ is reducing for $A$, then $\Omega\YY$ is reducing for $B$.
	Similarly, let 	
	$\Omega':\HH_B\to\HH_A$ be an isometry such that $\Omega'\HH_B$ as well as $\HH_A\ominus \Omega'\HH_B$ reduce $A$ and $\Omega'B=A\Omega'$.

	Define then the map $\Psi$ by 
	\[
	\Psi(\YY)=\HH_A\ominus \Omega'(\HH_B\ominus \Omega(\YY)).
	\]
	Then $\Psi$ maps the complete lattice of subspaces of $\HH_A$ which are reducing with respect to~$A$ to itself, and it is monotone. By the Knaster--Tarski Theorem~\cite{T} it has therefore a fixed point $\YY_0$. If we define the  operator $W:\HH_A\to\HH_B$ by $W|\YY_0=\Omega$ and $W|\HH_A\ominus \YY_0=\Omega'{}^*$, then $W$ is a unitary operator and $W A=BW$.
\end{proof}

A simple example shows that the analogue of Theorem~\ref{th:solution for sim} for $\approx$ is not true in general. 

\begin{example}\label{ex:basic example} Denote by $S,Z$ the unilateral and bilateral shifts on the spaces $\ell^2_\bbN$ and $\ell^2_\bbZ$ respectively. 
Take $A=\oplus_{k=1}^\infty Z$, $B=A\oplus S$. Obviously $A\Red B$,  so $A\Prec B$. But it is also easy to show that $B\Prec A$, since $B$ is unitarily equivalent to the restriction of $A$ to the invariant subspace $\ell^2_\bbN\oplus \bigoplus_{k=2}^\infty \ell^2_\bbZ$. 
	
	Therefore $A\approx B$; on the other hand, they are not unitarily equivalent, since $A$ is unitary while $B$ is not.
	
\end{example}

 Therefore, in the case of $\approx$ we will be interested in classes of contractions for which the following is true:
 
 \begin{center}
 	($*$)\quad If $A\approx B$, then $A$ and $B$ are unitarily equivalent. 
 \end{center}

 We will see that in general we must suppose  a certain type of finite multiplicity.

\section{Completely nonunitary contractions}\label{se: n finite}

We  need some preliminaries concerning singular values of compact operators. For a compact operator $X$, we will denote by $\sigma_k(X)$  the singular values of $X$, arranged in decreasing order. First we state a classical inequality due to Horn~\cite{Ho}. 

\begin{lemma}\label{le:horn-lidski}
	If $X:\EE\to\FF$, $Y:\FF\to\GG$ are compact operators, then for any $k\ge 1$ we have
	\[
	\prod_{t=1}^{k} \sigma_t(YX)\le  \prod_{t=1}^{k} \sigma_t(Y)\prod_{t=1}^{k} \sigma_t(X).
	\]
\end{lemma}

In particular, the lemma is valid for finite rank operators which is the case of interest to us.

\begin{corollary}\label{co:consequence of horn}
		If $X:\EE\to\FF$ has finite rank and $Y:\FF\to\GG$ is a contraction, then for any $k\ge 1$ we have
		\[
		\prod_{t=1}^{k} \sigma_t(YX)\le  \prod_{t=1}^{k} \sigma_t(X).
		\]
\end{corollary}

\begin{proof}
	The result is obtained by applying Lemma~\ref{le:horn-lidski} to the finite rank operators $X$ and $Y|\RR(X)$.
\end{proof}

\begin{lemma}\label{le:lema cu minimax}
	Suppose that $X:\EE\to\FF$ has finite rank, $Y:\FF\to\GG$ is a  contraction, and $X$ and $YX$ have the same singular values. Then $\Range(X) \subset \DD_Y^\perp$.
\end{lemma}

\begin{proof}
	Let $r=\rank X=\rank YX$.  If $\tilde X=X|(\ker X)^\perp\to \Range(X)$ and $\tilde Y=Y|\Range(X)\to \Range(YX)$, then $\tilde{X}$ and $\tilde{Y}\tilde{X}$ are invertible operators between  spaces of dimension $r$, with nonzero singular values coinciding with those of $X$, respectively $YX$. Lemma~\ref{le:horn-lidski} gives
	\[
	\prod_{t=1}^{r} \sigma_t(\tilde Y\tilde X)\le \prod_{t=1}^{r} \sigma_t(\tilde Y) \prod_{t=1}^{r} \sigma_t(\tilde X).
	\]
	The hypothesis implies $\prod_{t=1}^{r} \sigma_t(\tilde Y)=1$; since $\tilde{Y}$ is a contraction, it must be unitary. But $\tilde Y=Y|\Range(X)$; so $Y|\Range(X)$ is isometric, which is equivalent to $\Range(X) \subset \DD_Y^\perp$.
\end{proof}

Our main result for c.n.u. contractions is the following theorem.

\begin{theorem}\label{th:general n fnite}
		Suppose $A, B$ are two c.n.u. contractions with $\min\{\dim\DD_A, \dim\DD_B\}<\infty$. If  $A\approx B$, then $A$ and $B$ are unitarily equivalent.
	\end{theorem}

\begin{proof}

By Lemma~\ref{le:dimensions defects}, $n:=\dim\DD_A=\dim\DD_B<\infty$. 
We will denote
\[
\begin{split}
\DD_{A^*}'&=\bigvee_{\lambda\in\bbD}\Theta_A(\lambda)\DD_A,\quad \DD_{A^*}''=\DD_{A^*}\ominus \DD_{A^*}',\\
\DD_{B^*}'&=\bigvee_{\lambda\in\bbD}\Theta_B(\lambda)\DD_B,\quad \DD_{B^*}''=\DD_{B^*}\ominus \DD_{B^*}'.
\end{split}
\] 
The statement of the theorem is symmetric in $A$ and $B$, so we will assume in the sequel that $\dim\DD_{A^*}''\ge\dim\DD_{B^*}''$.

  Since $A\Prec B$ means that $A$ is unitarily equivalent to the restriction of $B$ to an invariant subspace, Lemma~\ref{le:factorization for n finite} implies that  there are contractively valued analytic functions  $\Theta_1(\lambda), \Theta_2(\lambda)$, such that $\Theta_1(\lambda)$ coincides with $\Theta_A(\lambda)$, and 	$\Theta_B(\lambda)=\Theta_2(\lambda)\Theta_1(\lambda)$ is a regular factorization. Therefore there  
  are unitary operators $\omega, \omega_*$ such that
	$\Theta_B(\lambda)=\Theta_2(\lambda)\omega_*\Theta_A(\lambda)\omega$; note, in particular, that $\omega:\DD_B\to\DD_A$.
	
	So $\Theta_B(\lambda)\omega^*=\Theta_2(\lambda)\omega_*\Theta_A(\lambda)$, and, if we denote $\Theta'_B(\lambda)=\Theta_B(\lambda)\omega^* $ and $\Theta_3(\lambda)=\Theta_2(\lambda)\omega_* $, then 
	\begin{equation}\label{eq:basic}
	\Theta'_B(\lambda)=\Theta_3(\lambda)\Theta_A(\lambda).
	\end{equation}
	Remember that here $\Theta_A(\lambda):\DD_A\to \DD_{A^*}$, $\Theta_3(\lambda): \DD_{A^*}\to \DD_{B^*}$, $\Theta'_B(\lambda):\DD_A\to\DD_{B^*}$ coincides with $\Theta_B(\lambda)$, and the factorization~\eqref{eq:basic} is easily checked to be also regular.

	As noticed in Section~\ref{se:prelim}, it may happen that $\Theta_3$ is not pure, so we  write $\DD_{A^*}=\DD_{A^*}^p\oplus \DD_{A^*}^u$, $\DD_{B^*}=\DD_{B^*}^p\oplus \DD_{B^*}^u$, such that with respect to these decompositions we have $\Theta_3(\lambda)=\Theta_3^p(\lambda)\oplus W$, with $\Theta_3^p(\lambda)$ pure and $W$ unitary and constant.

	Fix $\lambda\in\bbD$.
	Applying Lemma~\ref{le:horn-lidski}, we obtain that for any $k=1,\dots, n$ we have 
		\[
		\prod_{t=1}^{k} \sigma_t(\Theta_B'(\lambda))\le \prod_{t=1}^{k} \sigma_t(\Theta_3(\lambda)) \prod_{t=1}^{k} \sigma_t(\Theta_A(\lambda)).
		\]
Since $\Theta_3(\lambda)$ is a contraction, the first product is at most 1, and thus	
\[
\prod_{t=1}^{k} \sigma_t(\Theta_B(\lambda))
=\prod_{t=1}^{k} \sigma_t(\Theta_B'(\lambda))
\le\prod_{t=1}^{k} \sigma_t(\Theta_A(\lambda)).
\]
A similar argument, 
using the opposite relation $B\Prec A$, leads to the reverse inequality, so in the end  we  obtain  that
\begin{equation}\label{eq:product equalities}
\prod_{t=1}^{k} \sigma_t(\Theta_B(\lambda))=\prod_{t=1}^{k} \sigma_t(\Theta_A(\lambda))
\end{equation}
for all $k=1,\dots, n$.

The first conclusion that follows from these relations is that the rank of $\Theta_A(\lambda)$ coincides with the rank of $\Theta_B(\lambda)$ (since it is the largest $k$ for which the product of the first $k$ singular values is nonzero). 
Then, by dividing the equalities~\eqref{eq:product equalities} for successive values of $k$, we obtain that $\Theta_A(\lambda)$ and $\Theta_B(\lambda)$ have the same singular values. The same is then true about $\Theta_A(\lambda)$ and $\Theta'_B(\lambda)$.

Apply then Lemma~\ref{le:lema cu minimax} to the factorization~\eqref{eq:basic}.
It follows  that the range of $\Theta_A(\lambda)$ is contained in the orthogonal of the defect of $\Theta_3(\lambda)$. But this last space is precisely $\DD_{A^*}^u$. Since it does not depend on $\lambda\in\bbD$,  we also have $\DD_{A^*}'\subset \DD^u_{A^*}$, whence $\DD^p_{A^*}\subset\DD_{A^*}'' $. Also, from~\eqref{eq:basic} it follows then that 
\[
\begin{split}
W(\DD_{A^*}')&=W\left( \bigvee_{\lambda\in\bbD}\Theta_A(\lambda)\DD_A \right)= \bigvee_{\lambda\in\bbD}W\Theta_A(\lambda)\DD_A=
\bigvee_{\lambda\in\bbD}\Theta_3(\lambda)\Theta_A(\lambda)\DD_A\\
&=
\bigvee_{\lambda\in\bbD}\Theta'_B(\lambda)\DD_A
=\bigvee_{\lambda\in\bbD}\Theta_B(\lambda)\DD_B
=\DD_{B^*}'.
\end{split}
\]

Suppose $\dim\DD_{A^*}''=\dim\DD_{B^*}''$.  We may then choose an arbitrary unitary operator $W':\DD''_{A^*}\to\DD''_{B^*}$ and define   $\Omega:=W'\oplus W:\DD_{A^*}\to\DD_{B^*}$.  
It follows  from~\eqref{eq:basic} that
\begin{equation*}\label{eq:theta B=omega theta A}
\Theta'_B(\lambda)=\Omega\Theta_A(\lambda), \text{ or } \Theta_B(\lambda)=\Omega\Theta_A(\lambda)\omega.
\end{equation*}
Therefore the characteristic functions $\Theta_A(\lambda)$ and $\Theta_B(\lambda)$ coincide, whence $A$ and $B$ are unitarily equivalent.

Since we have assumed that $\dim \DD_{A^*}''\ge \dim\DD_{B^*}''$, in order
to finish the proof of the theorem it is enough to show that we cannot have
 $\dim\DD_{A^*}''>\dim\DD_{B^*}''$. Suppose then  this is the case; in particular, $\dim\DD_{B^*}^p\le\dim\DD_{B^*}''<\infty$. Moreover, since
\[
W(\DD_{A^*}^u\ominus \DD_{A^*}' )=\DD_{B^*}^u\ominus \DD_{B^*}',
\]
we have $\dim\DD_{A^*}^p>\dim\DD_{B^*}^p $. Now relation~\eqref{eq:basic} translates as  
\begin{equation}\label{eq:basic 2}
\Theta'_B(\lambda)=\big(\Theta_3^p(\lambda)\oplus W\big) \big( 0\oplus \Theta_A^{r}(\lambda) \big),
\end{equation}
where $\Theta_A^{r}(\lambda):\DD_A\to \DD_{A^*}^u$ acts as $\Theta_A(\lambda)$.
By Lemma~\ref{le:regular properties} (ii), the regularity of the factorization~\eqref{eq:basic 2} implies the regularity of the factorization 
$
	\Theta_3^p(\lambda)\cdot 0,
$
where $0$ acts from $\{0\}$ to $\DD_{A^*}^p$.
It follows then from Lemma~\ref{le:regular properties} (i)  
  that almost everywhere $\Theta_3^p(e^{it}):\DD_{A^*}^p\to \DD_{B^*}^p$ is an isometry. But this contradicts the inequality $\dim\DD_{A^*}^p>\dim\DD_{B^*}^p $.
The proof is therefore finished.
\end{proof}

Without the finite defect assumption,  it is not true, even for c.n.u contractions, that $A\approx B$ implies $A$ and $B$ unitarily equivalent, as shown by the counterexample below.
\begin{example}\label{ex:sum of Jordan}
	For $m\in\bbN$, denote by $S_m$ the operator having as matrix the Jordan cell of eigenvalue 0 and dimension precisely~$m$ (in particular, $S_1$ is the zero operator acting on $\bbC$). So, if $H_m$ is a Hilbert space of dimension~$m$ having as basis $e^m_1, \dots, e^m_m$, then $S_me^m_k=e^m_{k+1}$ for $k\le m-1$ and $S_me^m_m=0$.
	
	Define then $A=\bigoplus_{m=1}^\infty S_m$, and $B=\bigoplus_{m=2}^\infty S_m$, acting on $\HH_A=\bigoplus_{m=1}^\infty H_m$ and $\HH_B=\bigoplus_{m=2}^\infty H_m$ respectively.  It is obvious that $B\Prec A$, by the standard embedding of $\HH_B$ into $\HH_A$. 
	
	On the other hand, if we define $\Omega:\HH_A\to \HH_B$
	by $\Omega(e^m_k)=e^{m+1}_{k+1}$, then $\Omega$ is an isometry, $\Omega(\HH_A)$ is invariant with respect to $B$, and $\Omega B=A\Omega$; so $A\Prec B$.
	
	But $A$ and $B$ are not unitarily equivalent, since $H_1$ is a one-dimensional reducing subspace of $A$ contained in $\ker A$, while $\ker B=\bigvee_{m=2}^\infty \bbC e^m_m$ does not contain reducing subspaces. 
	
	It is worth noting that $A$ and $B$ are partial isometries, with both defect spaces of infinite dimension, so they belong to the class $\VV_\infty$ considered in~\cite{GMR}. Moreover, $A,B\in C_{00}$, so their characteristic functions are inner as well as *-inner. This is relevant in connection to~\cite[Section 9]{GMR}. Theorem 9.1 therein says that statement $(*)$ is true if $\Theta_A, \Theta_B$ are inner, with one-dimensional defects, while Theorem 9.2 claims, without including the proof, that the result is true in general. As a corollary of our Theorem~\ref{th:general n fnite}, we see that it is indeed true for finite dimensional defect spaces, while the above example shows that it cannot be extended to the infinite dimensional situation. 
\end{example}

\section{General contractions}\label{se:general}

We will investigate in this section statement $(*)$ for other classes of contractions. The next lemma is  elementary.

\begin{lemma}\label{le:invariant is reducing for unitary}
	If $\YY\subset\HH_B$ is a closed subspace, $B\YY\subset \YY$, and $B|\YY$ is unitary on $\YY$, then $\YY$ is reducing for $B$.
\end{lemma}

\begin{proof}
	If $P$ is the ortogonal projection onto the invariant subspace $\YY$ and $A=B|\YY$, then $A^*=PB^*|\YY$. But if $A$ is unitary, then $\|A^*x\|=\|x\|$ for all $x\in \YY$, so
	\[
	\|x\|= \|A^*x\|= \|PB^*x\|\le \|B^*x\|\le \|x\|.
	\]
	So the inequality is actually an equality, $PB^*x=B^*x$, and $\YY$ is also invariant to $B^*$.
\end{proof}

A few consequences for the preorder  relations are gathered in the next corollary.

\begin{corollary}\label{co:unit cnu} 
	{\rm (i)}
	If $A\Prec B$ and $A$ is a unitary operator, then $A\Red B$. 
	
	{\rm (ii)} If $A\Prec B$, then $A_u\Red B_u$. 
	
	{\rm (iii)}If $A\approx B$, then $A_u\sim B_u$.
\end{corollary}

\begin{proof}
(i) follows immediately from Lemma~\ref{le:invariant is reducing for unitary}. For (ii), if $A\Prec B$, then $A_u\Prec B$, and therefore $A_u\Red B$ by (i). Obviously a subspace that reduces $B$ to a unitary operator is contained in $\HH_B^u$, so then $A_u\Red B_u$. Finally, (iii) is an immediate consequence.
\end{proof}

In particular, Corollary~\ref{co:unit cnu} combined with Theorem~\ref{th:solution for sim} show that statement $(*)$ in Section~\ref{se:main problem} is true if $A,B$ are unitary operators. Unitary equivalence of unitary operators is completely described by multiplicity theory; see, for instance,~\cite{Ha}. We will not discuss it further here, but we want to point out a consequence that will be used below.

\begin{lemma}\label{le:unitary multiplicity}
	Suppose $U\in\LL(\HH_U)$ is a unitary operator of finite multiplicity (that is, $U$ is a finite direct sum of cyclic unitary operators). If $\YY\subset\HH_U$ is a reducing subspace for $U$ and $U|\YY$ is unitarily equivalent to $U$, then $\YY=\HH_U$.
\end{lemma}


Statement $(*)$ is also true for certain classes of more general contractions.
As in the case of c.n.u. contractions, some finite multiplicity condition is necessary. We start with a lemma.

%

\begin{lemma}\label{le:order extends to dilations}
	Suppose $A\Prec B$, implemented by $\Omega$. If $U,V$ are the minimal unitary dilations of $A, B$ respectively, then $U\Red V$. Moreover, if $V\in\LL(\KK_B)$, we may choose $\KK_A$ and an implementing isometry $\tilde{\Omega}:\KK_A\to \KK_B$ such that $\tilde{\Omega} x=\Omega x$ for all $x\in\HH_A$.
\end{lemma}

\begin{proof} We may suppose that $\HH_A\subset \HH_B$ is a subspace invariant with respect to $B$ and $A=B|\HH_A$; thus
	 $\Omega$ is the embedding of $\HH_A$ into $\HH_B$. If $V\in\LL(\KK_B)$ is the minimal unitary dilation of $B$, then $B^n=P^{\KK_B}_{\HH_B}V^n|\HH_B$ for all $n\in \bbN$. 
	
	Define then $\KK_A=\bigvee_{k=-\infty}^\infty V^k \HH_A$. Then $\KK_A$ reduces $V$; if $x\in\HH_A$, then, for all $n\in\bbN$
	\[
	P^{\KK_A}_{\HH_A}V^n x= P^{\KK_B}_{\HH_A} V^n x=  P^{\HH_B}_{\HH_A} P^{\KK_B}_{\HH_B} V^n x=P^{\HH_B}_{\HH_A} B^n x= A^nx.
	\]
 Therefore $U:=V|\KK_A$ is a minimal unitary dilation of $A$. Obviously $U\Prec V$, which by Corollary~\ref{co:unit cnu} implies $U\Red V$. Moreover, the implementing map $\tilde{\Omega}$ is the inclusion of $\KK_A$ into $\KK_B$. Restricted to $\HH_A$, this becomes  the inclusion of $\HH_A$ into $\KK_B$ (whose image is actually contained in $\HH_B\subset \KK_B$); so $\tilde{\Omega} x=\Omega x$ for all $x\in\HH_A$.
\end{proof}

The next result describes a rather general situation in which statement $(*)$ in Section~\ref{se:main problem} is true. 

\begin{theorem}\label{th:a more general result}
	Suppose $A\approx B$, and $A$ has the following two properties:
	\begin{itemize}
		\item[(i)] The absolutely continuous unitary part of $A$ has finite multiplicity.
		\item[(ii)] $\dim\DD_A<\infty$.
	\end{itemize}
	Then $A$ and $B$ are unitarily equivalent.
\end{theorem}

\begin{proof}
	Consider the decomposition $\HH_A=\HH_A^{c}\oplus\HH_A^{a}\oplus \HH_A^s$, which reduces $A$ to its completely nonunitary, absolutely continuous (with respect to Lebesgue measure) unitary, and singular unitary parts, respectively.
	According to Corollary~\ref{co:unit cnu}, $A_u$ is unitarily equivalent to $B_u$, and therefore $A_a, B_a$ are also unitarily equivalent. 
	
	Suppose $\Omega:\HH_A\to \HH_B$ is an isometry that implements  $A\Prec B$; that is $\Omega A=B\Omega$. Then $ \Omega \HH_A^a$ is a subspace of $\HH_B$ invariant with respect to $B$ and such that $B| \Omega \HH_A^a$ is unitarily equivalent to $A_a$; it also reduces $B$ by Lemma~\ref{le:invariant is reducing for unitary} to an absolutely    continuous unitary operator, whence $\Omega\HH_A^a\subset \HH_B^a$.  Then assumption (i) and Lemma~\ref{le:unitary multiplicity} imply that $\Omega\HH_A^a= \HH_B^a$. 
	So 
	\[
\Omega(\HH_A^{c}\oplus \HH_A^s)=	\Omega(\HH_A\ominus\HH_A^a )\subset\HH_B\ominus \HH_B^a= \HH_B^{c}\oplus \HH_B^s
	\]
	  is invariant to $B$ and implements the relation $A_c\oplus A_s\Prec B_c\oplus B_s$.

	Consider now the minimal unitary dilations $U\oplus A_s\in\LL(\KK_A'\oplus \HH_A^s)$ and $V\oplus B_s\in\LL(\KK_B'\oplus \HH_B^s)$ of $A_c\oplus A_s$ and $B_c\oplus B_s$ respectively. By Lemma~\ref{le:order extends to dilations}, $U\oplus A_s\Prec V\oplus B_s $ and there is an isometry $\tilde{\Omega}:\KK_A'\oplus \HH_A^s\to \KK_B'\oplus \HH_B^s $ which extends $\Omega|\HH_A^{c}\oplus \HH_A^s$, such that $\tilde{\Omega}(U\oplus A_s)= (V\oplus B_s)\tilde{\Omega} $. 
	
	But $U$, being the minimal unitary dilation of a c.n.u. contraction, is absolutely continuous. Therefore $\tilde{\Omega}(\KK_A')\subset \KK_B'$. By Lemma~\ref{le:order extends to dilations}, $\Omega=\tilde{\Omega}|\HH_A^c\oplus \HH_A^s$, whence $\Omega(\HH^c_A)\subset \KK'_B\cap \HH_B=\HH^c_B$. Moreover, from the equality
	 $\Omega(A_c\oplus A_s)=(B_c\oplus B_s)\Omega$ it follows then that $\Omega A_c=B_c\Omega$, so $\Omega|\HH^c_A$ implements the relation
	 $A_c\Prec B_c$.
	
	The hypothesis of the theorem being symmetric in $A,B$, the  roles of $A$ and $B$ can be interchanged, whence $A_c\approx B_c$. As $\dim\DD_{A_c}=\dim\DD_A<\infty$, we may apply Theorem~\ref{th:general n fnite} to conclude that $A_c$ and $B_c$ are unitarily equivalent. Since at the beginning of the proof we had noted that $A_u$ and $B_u$ are also unitarily equivalent, the same is then true about $A=A_c\oplus A_u$ and $B=B_c\oplus B_u$.
\end{proof}
 
 Examples~\ref{ex:basic example} and~\ref{ex:sum of Jordan} show that both  conditions (i) and (ii) in the statement of Theorem~\ref{th:a more general result} are necessary to ensure the truth of statement $(*)$.
 
 \section*{Acknowledgements}

The author thanks the referee for his careful reading and remarks, which lead to a significant improvement of the paper.
 
%

\end{document}